\documentclass{ws-ijm}
\usepackage{cite}
\usepackage{hyperref}
\usepackage{wasysym}

\usepackage{amsmath}
\usepackage{amssymb}
\usepackage{tikz}

\usepackage{enumitem}
\newcounter{my_enumerate_counter}
\newcommand{\pushcounter}{\setcounter{my_enumerate_counter}{\value{enumi}}}
\newcommand{\popcounter}{\setcounter{enumi}{\value{my_enumerate_counter}}}

\newcommand{\qed}{\null\nobreak\hfill\ensuremath{\square}}%

\newcommand{\bb}{\mathbf b}

\newcommand{\bbF}{{\mathbb F}}

\newcommand{\bbZ}{{\mathbb Z}}
\newcommand{\bbT}{\mathbb T}

\newcommand{\bbN}{{\mathbb N}}
\newcommand{\bbC}{\mathbb C}

\newcommand{\cR}{{\mathcal R}}

\newcommand{\cZ}{{\mathcal Z}}

\newcommand{\rs}{\restriction}

\newcommand{\cC}{\mathcal C}
\newcommand{\cF}{\mathcal F}

\newcommand{\cG}{\mathcal G}

\newcommand{\cB}{\mathcal B}
\newcommand{\cK}{\mathcal K}

\newcommand{\sfC}{\mathsf C}

\newcommand{\e}{\varepsilon}

\newtheorem{THM}{Theorem}
\newtheorem{claim}[theorem]{Claim}

\DeclareMathOperator{\diag}{diag}

\newcommand{\cstar}{$\mathrm{C}^*$}
\newcommand{\cst}{\mathrm{C}^*}

 \numberwithin{equation}{section}

 \DeclareMathOperator{\Ell}{Ell}
\renewcommand\thesubsubsection{\thesubsection (\alph{subsubsection})}
\renewcommand{\phi}{\varphi}

\renewenvironment{abstract}
 {\small
  \begin{center}
  \bfseries \abstractname\vspace{-.5em}\vspace{0pt}
  \end{center}
  \list{}{
    \setlength{\leftmargin}{.5cm}%
    \setlength{\rightmargin}{\leftmargin}%
  }%
  \item\relax}
 {\endlist}

\DeclareMathOperator{\CCR}{A}

\begin{document}

\markboth{I. Farah and N. Manhal}{Nonseparable CCR algebras}

\catchline{}{}{}{}{}

\title{Nonseparable CCR algebras}

\author{Ilijas Farah}


%

\address{Department of Mathematics and Statistics,
York University,
4700 Keele Street,
Toronto, Ontario, Canada, M3J
1P3} 
\address{Matemati\v cki Institut SANU\\
Kneza Mihaila 36\\
11\,000 Beograd, p.p. 367\\
Serbia\\
email: ifarah@yorku.ca\\
http://www.math.yorku.ca/$\sim$ifarah}

\author{Najla Manhal}

%

\address{Department of Mathematics and Statistics,
York University,
4700 Keele Street,
Toronto, Ontario, Canada, M3J
1P3\\
najla.muhee@gmail.com}

\maketitle

\begin{abstract}
Extending a result of the first author and Katsura, we prove that for every UHF algebra $A$ of infinite type, in every uncountable cardinality $\kappa$ there are $2^\kappa$ nonisomorphic approximately matricial C*-algebras with the same $K_0$ group as $A$. These algebras are group \cstar-algebras `twisted' by prescribed canonical commutation relations (CCR), and they  can also be considered as nonseparable generalizations of noncommutative tori. 
\end{abstract}



\maketitle

The inspiration for  the research presented here  comes from a question of Dixmier (\cite{Dix:Some}), who asked whether Glimm's result that every separable, unital,  AM (also called matroid) \cstar-algebras is UHF  (\cite{Glimm:On}) extends to the nonseparable case. (A \cstar-algebra is \emph{approximately matricial}, or AM, if it is an inductive limit of full matrix algebras. It is \emph{uniformly hyperfinite}, or UHF, if it is a tensor product of full matrix algebras.) The question stated in this form was answered in \cite{FaKa:Nonseparable}. The main result of Glimm's paper was classification of separable UHF algebras using a smooth invariant that we now describe. For every separable, unital, AM algebra $A$ there exists a sequence $k(j)\in \bbN\cup \{\infty\}$, for $j\in \bbN$, such that (with $p(j)$, for $j\in \bbN$ being an enumeration of the primes and $M_{p^\infty}(\bbC)$ denoting $\bigotimes_\bbN M_p(\bbC)$)
\[
\textstyle A\cong \bigotimes_j M_{p(j)^{k(j)}}(\bbC). 
\]
Thus the sequence $(k(j)\vert j\in \bbN)$, identified with the generalized integer (also called `supernatural number')
\[
\textstyle n_A=\prod_j p(j)^{k(j)}
\]
is a complete isomorphism invariant for separable, unital, AM algebras.   
A generalized integer $n$ corresponds to the $K_0$ group of $A$, equal to (with the convention that $l>0$ divides $n=\prod_j p(j)^{k(j)}$ if for every prime $p(j)$, $p(j)^k$ divides $l$ implies that $k\leq k(j)$) 
\[
\bbZ[1/n]=\{k/l\vert k\in \bbZ,l> 0, l\text{ divides } n\}. 
\]
  In \cite{FaKa:NonseparableII} it was shown that for every uncountable cardinal $\kappa$ there are~$2^\kappa$ nonisomorphic AM algebras of density character $\kappa$ with the same $K_0$ as the CAR algebra~$M_{2^\infty}$. Since AM algebras are monotracial and have trivial $K_1$ groups, these algebras have the same Elliott invariant as $M_{2^\infty}$ (see e.g., \cite{Ror:Classification}).  These algebras were universal \cstar-algebras given by generators and relations. The relations were coded by a graph, in which the vertices corresponded to the generators (self-adjoint unitaries)  while the edges determined which of the generators anticommute (i.e., satisfy the relation $vw=-wv$). 

The question whether the analogous result can be proven for other UHF algebras in place of $M_{2^\infty}$  remained open. Resolving it required a convenient coding of canonical commuting relations between non-self-adjoint unitaries. 

\begin{THM} \label{T.1} Suppose that $A$ is a UHF algebra  such that $p^\infty$ divides $n_A$ for some prime $p$. Then for every uncountable cardinal $\kappa$ there are $2^\kappa$ nonisomorphic AM algebras of density character $\kappa$  with the same $K_0$ group, and even the same Elliott invariant,\footnote{For the definitions of $K_0$ and  the Elliott invariant see e.g., \cite{Ror:Classification}. It will  be used only in Lemma~\ref{L.K-theory}, and even there only implicitly.}   as $A$.   
\end{THM}

In the proof of this theorem we use twisted group \cstar-algebras associated with canonical commutation relations (CCR). These algebras can also be considered as generalizations of noncommutative tori (see e.g., \cite[\S II.10.7.5]{Black:Operator}, \cite{rieffel1990non},  or \cite{phillips2006every}).

\section{Canonical commutation relations} \label{S.CCR} 
In this section we introduce \cstar-algebras given by canonical commutation relations (CCR, not to be confused with completely continuous representations). Suppose that $X$ is a set and 
\begin{equation}\label{Eq.Gamma}
\Gamma=\bigoplus_{\xi\in X} C_\xi
\end{equation}
 is a direct sum of cyclic groups. Such $\Gamma$ can be coded by a function $f_\Gamma\colon X \to \{2,3,\dots, \aleph_0\}$ such that $f_\Gamma(\xi)=|C_\xi|$ for all $\xi$. Thus $\Gamma$ can be presented as 
\[
\langle g(\xi): \xi\in X \vert g(\xi)^{f_\Gamma(\xi)}=1, g(\xi)g(\eta)=g(\eta)g(\xi): \xi\in X, \eta\in X\rangle. 
\]
Throughout this paper, we adopt the convention that a group $\Gamma$ has a decomposition into a direct sum of cyclic groups as in \eqref{Eq.Gamma}, and that $g(\xi)$ denotes a fixed generator of $C_\xi$. 

\begin{definition} \label{Def.CCR.triple} 
A \emph{CCR triple} is a triple $(X,\Gamma,\Theta)$ where $X$ is a set with a fixed linear (i.e., total) ordering, $\Gamma=\bigoplus_{\xi\in X} C_\xi$ is a direct sum of cyclic groups, and $\Theta\colon X^2\to \bbT$ is such that for all $\xi$ and $\eta$ in $X$ the following requirements are met.  	
\begin{enumerate}[label=(CCR.\arabic*), leftmargin=5\parindent]
\item\label{B.1} $\Theta(\xi,\xi)=1$. 
\item \label{B.2} $\Theta(\xi,\eta)=\overline{\Theta(\eta,\xi)}$. 
\item \label{B.4} If $g(\xi)^m=1$ and $g(\eta)^n=1$ then $\Theta(\xi,\eta)^{\gcd(m,n)}=1$, where $\gcd(m,n)$  denotes the greatest common divisor of $m$~and~$n$. 	
\pushcounter
\end{enumerate}
\end{definition}
The index-set $X$ is included in the CCR triple $(X,\Gamma,\Theta)$ as a reminder that $\Gamma$ is taken with a fixed decomposition as in \eqref{Eq.Gamma}. The linear ordering on $X$ is used only in Lemma~\ref{L.CCR.1} and it will be suppressed throughout. Typically, $X$ will be a cardinal, a subset of a cardinal, and in any case it will be equipped with a natural linear ordering.     

The function $\Theta$ gives canonical commutation relations on $\Gamma$. In the literature (\cite[\S II.10.7.5]{Black:Operator}, \cite{rieffel1990non},  or \cite{phillips2006every}), $\Theta$ is usually given indirectly, by a skew-symmetric matrix $\theta$, so that $\Theta(\xi,\eta)=\exp( \pi i \theta(\xi,\eta))$. In the existing literature, $X$ is  finite  and the order of each generator $g(\xi)$ is assumed to be infinite. This data gives a noncommutative torus. In this situation, condition \ref{B.4} is unnecessary and the following example is included here to justify it. 

\begin{example}Suppose that unitaries $u$ and $v$ in a \cstar-algebra $A$ satisfy $u^m=1=v^n$ and $uv=\lambda vu$. Then $\lambda^{\gcd(m,n)}=1$.  
To see this, note that by induction we have $u^m v =\lambda^m v u^m$ and therefore $v=\lambda^m v$, hence $\lambda^m=1$. A proof that $\lambda^n=1$ is analogous. Since $\gcd(m,n)$ is an integral linear combination of $m$ and $n$, the conclusion follows. 
\end{example}

 In Proposition~\ref{P.existence} we will prove that  conditions \ref{B.1}--\ref{B.4} are sufficient for the existence of a CCR algebra $\CCR_{(X,\Gamma,\Theta)}$, in the sense that these relations are satisfied by some choice of unitaries on a Hilbert space of the appropriate order. In order to prove this, we will need the following (well-known) lemma. 

\begin{lemma} \label{L.lambda} Suppose $\lambda\in \bbT$ and $n\geq 2$. If $\lambda$ is a primitive $n$th root of unity,  then $M_n(\bbC)$ is generated by unitaries $v_n$ and $w_n$  such that $v_n w_n =\lambda w_n v_n$ and each one of $v_n$ and $w_n$ has order $n$. 

If $\lambda^n\neq 1$ for all $n\geq 1$, there are unitaries of infinite order $v_n$ and $w_n$ on a separable Hilbert space such that $v_n w_n=\lambda w_n v_n$. 
\end{lemma}

\begin{proof} Take $w_n=\diag(1,\lambda, \dots, \lambda^{n-1})$ and take $v_n$ to be the permutation unitary matrix such that $(v_n)_{i,i+1}=1$ for $1\leq i<n$ and $(v_n)_{n,1}=1$, with all other entries equal to 0. Then (with $\xi_i$, for $1\leq i\leq n$ denoting the standard basis of $\bbC^n$) we have $v_n \xi_i=\xi_{i+1}$ if $i<n$ and $v_n \xi_n=\xi_1$. Thus $v_n w_n v_n^*(\xi_i)=\lambda^i \xi_i$, and $v_n w_n v_n^*=\lambda w_n$. It remains to prove that $u_n$ and $v_n$ generate $M_n(\bbC)$. We will prove that $u^k v^l$, for $0\leq k,l<n$, are linearly independent. Otherwise, there are $k<n$ and $l<n$ such that $u^kv^l=\sum_{i<n, j<n} \lambda_{i,j} u^iv^j$, with $\lambda_{k,l}=0$. By multiplying with $(u^*)^k$ on the left and $(v^*)^l$ on the right, we obtain $1=\sum_{i<n, j<n} \lambda_{i,j} u^iv^j$, with $\lambda_{0,0}=0$. With $\tau$ denoting the unique tracial state on $M_n(\bbC)$, we have $\tau(u^iv^j)\neq 0$ if and only if $i=j=0$; contradiction. Therefore $\cst(u,v)$ has dimension (at least) $n^2$, and it is therefore equal to $M_n(\bbC)$. 

Now suppose that $\lambda^n\neq 1$ for all $n$. In this case we use the generators of the \emph{irrational rotation algebra} associated with $\lambda$ (see e.g., \cite[\S II.8.3.3~((i)]{Black:Operator}). More precisely, let $v_{\aleph_0}$ denote the standard generator of $C(\bbT)$ (i.e., $v_{\aleph_0}$ is a unitary whose spectrum is equal to $\bbT$). Let $\alpha$ be the rotation of $\bbT$ by angle $\theta$,  chosen so that $\lambda=e^{i\theta}$. In the crossed product $C(\bbT)\rtimes_\alpha \bbZ$, let $u_{\aleph_0}$ be the unitary that implements $\alpha$. Then $u_{\aleph_0}v_{\aleph_0}u_{\aleph_0}^*= \lambda v_{\aleph_0}$, as required.  Since the irrational rotation algebras are simple, any nondegenerate representation of $C(\bbT)\rtimes_\alpha \bbZ$ on a Hilbert space yields $u_{\aleph_0}$ and $v_{\aleph_0}$ as required. 
\end{proof}

The proof of the following is based on the proof of \cite[Proposition~10.1.3 (1)]{Fa:STCstar}).

\begin{proposition}\label{P.existence}
	For every CCR triple $(X,\Gamma,\Theta)$, there exists a universal \cstar-algebra $\CCR_{(X,\Gamma,\Theta)}$ given  by generators  $u_\xi$, for $\xi\in X$  and the following relations (the order function $f_\Gamma$ was defined after \eqref{Eq.Gamma}):  
\begin{multline*}
\cR(\Gamma,\Theta)=\{u_\xi u_\xi^*=1=u_\xi^* u_\xi, u_\xi u_\eta=\Theta(\xi,\eta) u_\eta u_\xi\vert \xi\in X,\eta\in X\}\\
\cup \{u_\xi^n=1\vert \xi\in X, n=f_\Gamma(\xi)\text{ is finite}\}
\end{multline*}
In this algebra, the order of the unitary $u_\xi$ is equal to the order of the generator $g(\xi)$ in $\Gamma$. 
\end{proposition}

\begin{proof} Since all of the generators are unitaries (and therefore of norm~1), it suffices to prove that every $F\Subset \cR(\Gamma,\Theta)$\footnote{By $F\Subset X$ we abbreviate the assertion that $F$ is a finite subset of $X$.}  is represented in some \cstar-algebra such that for every $\xi$ such that $u_\xi$ occurs in $F$ the order of $u_\xi$ is equal to the order of $g(\xi)$  (this is well-known, see e.g., \cite[Lemma~2.3.11]{Fa:STCstar}). Let $G=\{\xi\vert g(\xi)$ occurs in $F\}$. We can re-enumerate $G$ and identify it with   $n=\{0,\dots, n-1\}$,  with $n=|G|$. Let   $\Gamma'=\prod_{i<n} C_i$. 

For $i<j<n$, let $H_{i,j}$ be the Hilbert space of dimension (with  $\aleph_0=\min\emptyset$)
\[
d(i,j)=\begin{cases}
\min\{d\vert \Theta(g(i), g(j))^d=1\}, \text{ if } i<j.\\
\min\{d\vert g(i)^d=1\}, \text{ if } i=j. 	
\end{cases}
\]
Let $H=\bigoplus_{0\leq i\leq j<n} H_{ij}$. Fix $k<n$. For $i<j<n$ we will define a unitary $u_{i,j,k}$ on  $H_{i,j}$ as follows (using the notation from Lemma~\ref{L.lambda})
\[
u_{i,j,k}=
\begin{cases} 
	1, \text{ if } k\notin \{i,j\},\\
	v_{d(i,j)}, \text{ if $k=i$ and $i<j$}\\
	w_{d(i,j)}, \text{ if $k=j$ and $i<j$},\\
	v_{d(i,j)}, \text{ if $i=j=k$}.
\end{cases}
\]
Let $u_k=\bigoplus_{i<j<n} u_{i,j,k}$. This is a unitary whose order is the least common multiple of the orders of $u_{i,j,k}$, for $i<j<n$. These orders range over the orders of $\Theta(g(k), g(l))$, for $l<n$, and the order of $g(k)$ (in the case when $i=j=k$). Therefore \ref{B.4} implies that $u_k$ is a unitary of the same order as $g(k)$. 

\begin{claim} We have  $u_k u_l =\Theta(g(k),g(l))u_l u_k$ for all $k<n$ and $l<n$. 
\end{claim} 

\begin{proof} If $k=l$ then this is a consequence of \ref{B.1}, $\Theta(g(k),g(k))=1$. Assume $k\neq l$.  Fix for a moment $i<j<n$. If $\{i,j\}\neq \{k,l\}$, then a glance at the definitions of $u_{i,j,k}$ and $u_{i,j,l}$ confirms that $u_{i,j,k}$ and $u_{i,j,l}$ commute. 
Therefore for $k<l$ we have that $u_k u_l=\lambda u_l u_k$ if and only if $u_{k,l,k} u_{k,l,l} = \lambda u_{k,l,l} u_{k,l,k}$.   

Fix $i<j$ and suppose that  $\{i,j\}=\{k,l\}$.  If  $k<l$, then the choice of $u_{i,j,k}$ and $u_{i,j,l}$ implies that $u_{i,j,k} u_{i,j,l}=\Theta(g(k), g(l)) u_{i,j,l} u_{i,j,k}$. Similarly, if $k>l$, then the definition  and  \ref{B.2} imply  
\[
u_{i,j,k} u_{i,j,l}=\overline{\Theta(g(l), g(k))} u_{i,j,l} u_{i,j,k}=\Theta(g(k), g(l)) u_{i,j,l} u_{i,j,k}.
\]
 By the previous paragraph, this completes the proof. 
\end{proof}

This completes the proof that every $F\Subset \cR(\Gamma,\Theta)$ is represented  in some \cstar-algebra. By \cite[Lemma~2.3.11]{Fa:STCstar}, there exists a universal \cstar-algebra given  by generators $\cG(\Gamma)$ and relations $\cR(\Gamma,\Theta)$.  By  construction, the order of $u_\xi$ is equal to the order of $g(\xi)$ for all $\xi$. 
\end{proof}

Suppose that $(X,\Gamma,\Theta)$ is a CCR triple.   We associate elements of $\Gamma$ with finite products of the form 
\begin{equation}\label{Eq.g}
\textstyle\prod_\xi g(\xi)^{m(\xi)}
\end{equation}
 where $0\leq m(\xi)<f_\Gamma(\xi)$ for all $\xi$ and $m(\xi)=0$ for all but finitely many~$\xi$. For $g$ as in \eqref{Eq.g}, in $\CCR_{(X,\Gamma,\Theta)}$ consider the unitary 
 \begin{equation}	\label{Eq.ug}
 \textstyle u_g=\prod_{\xi} u_\xi^{m(\xi)}
 \end{equation}
(in this product the generators $u_\xi$ are taken in the order increasing in the previously fixed linear ordering of $X$). 
In Lemma~\ref{L.CCR.1} we show how to extend $\Theta$ to a bicharacter on~$\Gamma$ and that this bicharacter (also denoted $\Theta$) controls the commutation on all $u_g$, for $g\in \Gamma$.  

\begin{lemma} \label{L.CCR.1} Suppose that $(X,\Gamma,\Theta)$ is a CCR triple. Extend $\Theta$ to $\Gamma$ by letting  
\[
\textstyle \Theta(\prod_\xi g(\xi)^{m(\xi)},\prod_\eta g(\eta)^{(n(\eta)})
=\prod_{\xi,\eta} \Theta(\xi,\eta)^{m(\xi)n(\eta)}. 
\]
Then the unitaries in $\CCR_{(X,\Gamma,\Theta)}$ defined in \eqref{Eq.ug} satisfy 
\[
u_g u_h =\Theta(u_g,u_h) u_h u_g.
\]
\end{lemma}

\begin{proof} The proof is a straightforward induction on the sum of the lengths of the words $g$ and $h$. 
\end{proof}

\begin{remark}`Twisted'  
\cstar-algebras associated to an abelian group and a bicharacter have been studied in \cite{slawny1972factor} and \cite{manuceau1968c}. Instead of CCR relations given by a bicharacter $\Theta$ and  
\begin{equation*}
	u_g u_h =\Theta(g,h) u_h u_g, 
\end{equation*}
as in our setting, these authors started from an appropriate  function $\bb\colon \Gamma^2\to \bbT$ and considered the relations
\begin{equation*}
	u_g u_h =\bb(g,h) u_{hg}.  
\end{equation*}
As pointed out e.g., in \cite[Remark~1.2]{phillips2006every} or \cite[Section 4]{Rieff:Projective}, our $\CCR_{(X,\Gamma,\Theta)}$ can be recast as a \cstar-algebra given by $\Gamma$ and  $\bb$ and vice versa. In both \cite[Theorem~3.6]{slawny1972factor} and \cite[\S 2.2.3]{marcoux2013abelian} it was proven that if $\bb$ is nondegenerate (i.e., if $\bb(g,h)=1$ for all $h\in \Gamma$ implies $g=e$) then the algebra associated to $\Gamma$ and $\bb$ is simple. Using this result would have shortened our arguments somewhat, but we chose to give a self-contained presentation.   
\end{remark}

\section{Functoriality} 
\label{S.Funct}

Consider the category of CCR triples $(X,\Gamma,\Theta)$. The morphisms in this category,  $\varphi\colon (X,\Gamma,\Theta)\to (X',\Gamma',\Theta')$, are injective homomorphisms  from $\Gamma$ into $\Gamma'$ which preserve  the CCR relations (as extended to the entire group in Lemma~\ref{L.CCR.1}) i.e., $\Theta'(\varphi(g),\varphi(h))=\Theta(g,h)$, for all $g$~and~$h$ in $\Gamma$. 
 We will prove that the universal construction $\CCR_{(X,\Gamma,\Theta)}$ is functorial in the subcategory of CCR triples associated with locally finite groups.\footnote{A group is called \emph{locally finite} if every finitely generated subgroup is finite. For abelian groups this is equivalent to every element having finite order.}

\begin{proposition}\label{P.funct} If $\Gamma$ is a locally finite group and $(X,\Gamma,\Theta)$ and  $(X',\Gamma',\Theta')$ are CCR triples, then every  morphism 
\[
\varphi\colon (X,\Gamma,\Theta)\to (X',\Gamma',\Theta')
\]
 gives a unique injective $^*$-homomorphism   $\Psi_\varphi\colon \CCR_{(X',\Gamma',\Theta')} \to \CCR_{(X,\Gamma,\Theta)}$ such that $\Psi_\varphi(u_g)=u_{\varphi(g)}$ for all $g\in \Gamma$. 
\end{proposition} 

The proof of Proposition~\ref{P.funct}  will be given after some preliminaries. 

\begin{definition} For a CCR triple $(X,\Gamma,\Theta)$, let $\bbC_\Theta\Gamma$  denote the  algebra of all finite linear combinations $\sum_g \lambda_g u_g$, with the multiplication `twisted' by $\Theta$, as 
\[
u_g u_h =\Theta(g,h) u_h u_g. 
\]
 This is a complex algebra dense in $\CCR_{(X,\Gamma,\Theta)}$. 
\end{definition}

The proof of Lemma~\ref{L.tracial} is based on the proofs of \cite[Lemma~1.5 and Theorem~1.9]{phillips2006every}, where its special case when $\Gamma$ is $\bbZ^n$ for some $n\geq 2$ was considered. An alternative proof of this lemma is analogous to the proof of \cite[Lemma~10.1.3]{Fa:STCstar}.

\begin{lemma}\label{L.tracial} Every CCR algebra $\CCR_{(X,\Gamma,\Theta)}$ has a  tracial state $\tau$ such that ($e$ denotes the identity element in $\Gamma$)
\begin{equation}\label{eq.tau}
\textstyle\tau(\sum_g \lambda_g u_g)=\lambda_e
\end{equation}
for all $\sum_g \lambda_g u_g$ in $\bbC_\Theta\Gamma$. If $\Gamma$ is locally finite, then $\tau$ is faithful. 
\end{lemma}

\begin{proof} 
Fix $F\Subset X$ and let  $\Gamma'=\bigoplus_{i\in F} C_i$. By re-enumerating we may assume $F=\{0,\dots, n-1\}$ for some $n\in \bbN$. For $m< n$ let $A_m=\cst(\{u_i\vert i\leq m\})$ and let~$\Gamma_m$ denote the subgroup of $\Gamma$ generated by $\{g(i)\vert i\leq m\}$.  Then $A_0$ is naturally isomorphic to $C(\hat \Gamma_0)$, where $\hat \Gamma_0$ is the Pontryagin dual of $\Gamma_0$. Let $\tau_0$ be the tracial state on $A_0$ such that $\tau(1)=1$ and $\tau_0(u_{g(0)})=0$. A simple Cauchy--Schwarz argument shows that this uniquely defines~$\tau_0$, and that it agrees with   \eqref{eq.tau}. By induction we will prove that for every $1\leq j\leq n$ there is a tracial state $\tau_j$ on $A_j$ such that $\tau_j\rs A_{j-1}=\tau_{j-1}$ and $\tau_j$ satisfies \eqref{eq.tau}.  

Suppose that $0\leq j<n$ and $\tau_j$ as required has been constructed so that it satisfies \eqref{eq.tau}. We will prove that  it is a tracial state.  The pertinent fact is that for every $g\in \Gamma$ we have $u_g u_{g^{-1}}=u_e = u_{g^{-1}} u_g$. For  $\sum_g \lambda_g u_g$ and $\sum_g \gamma_g u_g$ in $\bbC_{\Theta}\Gamma_j$ we have 
\[
\textstyle\tau_j( (\sum_g \lambda_g u_g)(\sum_h \gamma_h u_h))=\sum_{gh=e} \lambda_g \gamma_h 
=\tau( (\sum_h\gamma_h u_h) (\sum_g \lambda_g u_g)). 
\]
Since $\bbC_\Theta\Gamma_j$ is dense in $A_j$, this shows that $\tau_j$ is tracial. 

Conjugation by $u_{g(j+1)}$ induces an automorphism $\alpha=\alpha_{j+1}$ of $A_j$, since for every $h$  we have  $u_{g(j+1)} u_h u_{g(j+1)}^*=\Theta(g(j+1), h) u_h$. The algebra $A_{j+1}$ is therefore isomorphic to the crossed product $A_j\rtimes_{\alpha_j} C_j$ (recall that  $C_j$ is a, possibly finite,  cyclic group). A glance at \eqref{eq.tau} reveals that $\alpha$ is  $\tau_j$-preserving and therefore  we can find $\tau_{j+1}$ as required.  

We have proved that, for  every $F\Subset X$, with  $\Gamma'=\bigoplus_{\xi\in F} C_\xi$,  \eqref{eq.tau} defines a tracial state $\tau_{\Gamma'}$ on the \cstar-subalgebra $\cst(u_g\vert g\in \Gamma')$ of $\CCR_{(X,\Gamma,\Theta)}$. Since $\Gamma$ is a direct limit of its finitely generated subgroups and the tracial states $\tau_{\Gamma'}$ are compatible (in the sense that they agree on $\bbC_\Theta\Gamma$, when defined), this shows that \eqref{eq.tau} defines a tracial state $\tau$ on $\CCR(X,\Gamma, \Theta)$.  

It remains to verify that $\tau$ is faithful if $\Gamma$ is locally finite. Fix a finitely generated subgroup $\Gamma'$ of $\Gamma$. Since $\Gamma$ is locally finite, $\Gamma'$ is finite. Therefore every element of $A(\Gamma')=\cst(u_g\vert g\in \Gamma')$ is a finite linear combination of the canonical unitaries, and we can identify $A(\Gamma')$ with $\bbC_\Theta\Gamma'$.   To prove that $\tau$ is faithful on $A(\Gamma')$,  fix a nonzero  $\sum_g \lambda_g u_g$ in $\bbC_\Theta\Gamma'$. Then  (again using the fact that $u_g$ and $u_{g^{-1}}$ commute)
\[
\textstyle\tau_j( (\sum_g \lambda_g u_g)^*(\sum_g \lambda_g u_g))=\sum_g |\lambda_g|^2>0. 
\]
Therefore the restriction of $\tau$ to $A(\Gamma')$ is faithful for every finitely generated subgroup $\Gamma'$ of $\Gamma$. Let
\[
J=\{a\in \CCR_{(X,\Gamma,\Theta)}\vert \tau(a^*a)=0\}.
\]
 Since $\tau$ is a tracial state,  $J$ is a two-sided, norm-closed, ideal (\cite[Lemma~4.1.3]{Fa:STCstar}).  We have proved that   $\CCR_{(X,\Gamma,\Theta)}$ is an inductive limit of \cstar-subalgebras $A(\Gamma')$ for finite $\Gamma'\leq \Gamma$  such that the intersection of $J$ with each one of them is trivial. By a well-known property of \cstar-algebras (\cite[Proposition~2.5.3]{Fa:STCstar}), an ideal of an inductive limit is the inductive limit of its intersections with the algebras comprising the inductive system. Since $J\cap A(\Gamma')$ is trivial for all $\Gamma'$, this concludes the proof. 
\end{proof}

For a CCR triple $(X,\Gamma,\Theta)$ (with $\Gamma$ represented as a direct sum, as in \eqref{Eq.Gamma}) and $F\subseteq X$, we consider the group 
\begin{equation}\label{Eq.GammaF}
\Gamma_F=\bigoplus_{\xi\in F} C_\xi
\end{equation}
and the CCR triple $(F,\Gamma_F, \Theta_F)$.

\begin{lemma} \label{L.finite} \label{L.AF}  
Suppose that $(X,\Gamma,\Theta)$ is a CCR triple and $\Gamma$ is locally finite.  
\begin{enumerate}
\item \label{1.L.finite}  If $\Gamma$ is  finite then $A_{(X,\Gamma, \Theta)}$ is a finite-dimensional \cstar-algebra, and its dimension is $|\Gamma|$. 
\item\label{2.L.finite} The \cstar-subalgebra $\cst(u_g\vert g\in \Gamma_F)$ of $\CCR_{(X,\Gamma,\Theta)}$ is isomorphic to 
\[
A_F=\CCR_{(F,\Gamma_F,\Theta\rs \Gamma_F)}
\]
 for all $F\Subset X$. 
\item\label{3.L.finite}  $\CCR_{(X,\Gamma,\Theta)}$  is an AF algebra, and it is an inductive limit of the algebras $A_F$, for  $F\Subset X$.
\end{enumerate}
\end{lemma}

\begin{proof} \eqref{1.L.finite} Let $u_g$, for $g\in \Gamma$, be the unitaries as defined in Lemma~\ref{L.CCR.1}. Then $A_{(X,\Gamma,\Theta)}$ is the linear span of $u_g$, for $g\in \Gamma$, and therefore it is finite-dimensional of dimension at most $|\Gamma|$. For the converse inequality, we need to prove that $u_g$, for $g\in \Gamma$, are linearly independent. Assume otherwise, so that for some $g$ and scalars $\alpha_f$ we have  $u_g=\sum_{f\neq g} \alpha_f u_f$. By multiplying the equation with $u_g^*$, we obtain $u_e=\sum_{f\neq e} \beta_f u_f$. However, the tracial state $\tau$ defined in Lemma~\ref{L.tracial} satisfies $\tau(u_e)=1$ while $\tau(u_f)=0$ for all $f\neq e$; contradiction. 

\eqref{2.L.finite} In order to prove that $\cst(u_g\vert g\in \Gamma_F)\cong A_F$, it suffices to show that the $u_g$ are linearly independent. This follows by the argument using $\tau$ as in  \eqref{1.L.finite}. 

\eqref{3.L.finite} is an immediate consequence of \eqref{2.L.finite}. 
\end{proof}

In the \cstar-algebra $\CCR_{(X,\Gamma,\Theta)}$ associated to a CCR triple consider the inner product (with $\tau$ denoting the tracial state guaranteed by Lemma~\ref{L.tracial})
\[
\langle a, b\rangle_\tau=\tau(b^*a). 
\]
The completion of the algebra with respect to this sesquilinear form is of course the GNS Hilbert space corresponding to $\tau$, and the associated norm 
\[
\|a\|_{2,\tau}=\langle a,a\rangle^{1/2}
\]
 is dominated by the operator norm (because $\|\tau\|=1$). 

\begin{lemma} \label{L.ell2} In every CCR algebra  $A=\CCR_{(X,\Gamma,\Theta)}$ such that $\Gamma$ is locally finite the unitaries $u_g$, for $g\in \Gamma$, form an orthonormal basis for the pre-Hilbert space $(A,\langle\cdot, \cdot\rangle_\tau)$. 
\end{lemma}

\begin{proof} Let $\ell_2(\Gamma)$ denote the Hilbert space with the orthonormal basis $\delta_g$, for $g\in \Gamma$. The linear map that sends $\sum_j \lambda_j u_{g(j)}$ to $\sum_j \lambda_j \delta_{g(j)}$ is an isometry from  $(A,\|\cdot\|_2)$ into $\ell_2(\Gamma)$. This is because (using the fact that  $\tau(u_g)=0$ if $g\neq e$ and $\tau(u_e)=1$)
\[
\textstyle \|\sum_j \lambda_j u_{g(j)}\|_2^2
=\tau((\sum_j \lambda_j u_{g(j)})^*(\sum_j \lambda_j u_{g(j)}))=\sum_j |\lambda_j|^2. 
\]
Since these linear combinations are dense in $\ell_2(\Gamma)$,  this linear map extends to an isometry from the $\|\cdot\|_2$-completion of $(A,\langle\cdot, \cdot\rangle_\tau)$ onto $\ell_2(\Gamma)$, and  $\delta_g$, for $g\in \Gamma$, is an orthonormal basis for $\ell_2(\Gamma)$.  
\end{proof}

\begin{proof}[Proof of Proposition~\ref{P.funct}]
	Fix locally finite groups $\Gamma$ and $\Gamma'$ and a morphism $\varphi\colon (X,\Gamma,\Theta)\to (X',\Gamma',\Theta')$.  Let $u_g$, for $g\in \Gamma$, denote the canonical unitaries in $\CCR_{(X,\Gamma,\Theta)}$ and let $v_g=u_{g(\varphi(\xi))}$, for $g\in \Gamma$, denote the corresponding unitaries in $\CCR_{(X',\Gamma',\Theta')}$.	Consider the map  $u_{g}\mapsto v_{g}$ and let $\Psi_0$ be its extension to a linear map from $C_\Theta \Gamma$ into $C_{\Theta'}\Gamma'$. We need to prove that $\Psi_0$  extends to a $^*$-homomorphism   $\Psi_\varphi\colon \CCR_{(X,\Gamma,\Theta)}\to \CCR_{(X',\Gamma',\Theta')}$. Since $\Gamma$ is locally finite, $\CCR_{(X,\Gamma,\Theta)}$ is by Lemma~\ref{L.AF} an inductive limit of finite-dimensional subalgebras $A_F$, each of which is a linear span of $\{u_g\vert g\in \Gamma_F\}$, for $F\Subset X$. For $F\Subset X$ we have  $A_F\subseteq C_\Theta\Gamma$, hence the restriction of $\Psi_1$ to $A_F$ is a $^*$-homomorphism between \cstar-algebras, and  therefore has norm 1 (\cite[Lemma~1.2.10]{Fa:STCstar}). Since $g\neq g'$ implies $v_g\neq v_{g'}$, the restriction of $\Psi_1$ to $A_F$ is injective. By Lemma~\ref{L.AF}, $A_{(X,\Gamma,\Theta)}$ is the inductive limit of algebras $A_F$, and therefore $\Psi_1$ extends to a $^*$-homomorphism $\Psi\colon A_{(X,\Gamma,\Theta)} \to A_{(X',\Gamma',\Theta')}$.    
	
	The intersection of $\ker(\Psi)$ with $A_F$ is trivial for every finite $F$, and by \cite[Lemma~4.1.3]{Fa:STCstar}, $\ker(\Psi)$ is trivial. This completes the proof.  	
	\end{proof}

In the original version of this paper it was  asserted that $\delta_g$ is a Schauder basis for $A_{(X,\Gamma,\Theta)}$. The referee pointed out that our proof of this lemma was incomplete, but that it can be replaced with the use of orthonormal basis $\delta_g$, for $g\in \Gamma$, in $\ell_2(\Gamma)$. Proposition~\ref{P.Conditional} and Corollary~\ref{C.Conditional} serve to formalize this argument.  Analogous remarks apply to  \cite[Lemma~10.2.6 (1)]{Fa:STCstar}.

If $B$ is a \cstar-subalgebra of $A$, then a linear map  $E\colon B\to A$  is a \emph{conditional expectation} if   $E(bac)=bE(a)c$ for all $a\in A$ and all $b$ and $c$ in $B$. By Tomiyama's theorem,  this is equivalent to having $E(b)=b$ for all $b\in B$ and $\|E\|=1$ (see e.g.,  \cite{BrOz:C*}).

The following proposition is related to  \cite[Lemma~10.2.7]{Fa:STCstar}, where an analogous statement for graph CCR algebras, with an additional assumption that $G$ is finite, was proved. 

\begin{proposition}\label{P.Conditional}
	Suppose that $\Gamma$ is a locally finite group and $(X,\Gamma,\Theta)$ is a CCR triple. For $G\leq \Gamma$ consider the \cstar-subalgebra of $\CCR_{(X,\Gamma,\Theta)}$,\footnote{If $G=\Gamma_F$ as in \eqref{Eq.GammaF}, then $A(G)$ is $A_F$ as defined in the statement of Lemma~\ref{L.finite}.}  
	\[
	A(G)=\cst(u_g\vert g\in G). 
	\]
	For subgroups $G\leq H\leq \Gamma$ there exists a conditional expectation 
	\[
	E_{HG}\colon A(H) \to A(G). 
	\]
	This system of conditional expectations commutes: if $G\leq H\leq K\leq \Gamma$ then $E_{KG}=E_{HG}\circ E_{KH}$.  
\end{proposition}

\begin{proof} By Lemma~\ref{L.tracial}, $A=\CCR_{(X,\Gamma,\Theta)}$  has a faithful tracial state $\tau$.  We will use the notation $\ell_2(G)$ and $\langle \cdot, \cdot\rangle_\tau$ introduced in the proof of Lemma~\ref{L.ell2}.   The GNS space associated to $\tau$ is isomorphic to $\ell_2(\Gamma)$ and we will identify $A$ with its image $\pi_\tau[A]$ under the GNS representation. For $G\leq \Gamma$ let $\tau_G=\tau\rs G$. Writing $p_G$ for the projection from $\ell_2(\Gamma)$ onto $\ell_2(G)$, the representation 
\begin{equation}\label{Eq.piG}
\pi_G(a)= p_G \pi_\tau(a) p_G
\end{equation}
of $A(G)$ is equivalent to the GNS representation $\pi_{\tau_G}$ of $A(G)$ associated with~$\tau_G$. To see this, note that for $a\in C_{\Theta\rs G}G$ we have  $\tau_G(a)=\langle a\delta_e , \delta_e\rangle_\tau$ and that $\delta_e$ is a cyclic vector for $\pi_G[A(G)]$. Since $C_{\Theta\rs G} G$ is dense in $A(G)$, by the uniqueness of the GNS representation, $\pi_{\tau_G}$ is equivalent to $\pi_G$. 

For $G\leq  \Gamma$ let $\Xi_{G}\colon \pi_{G}[A(G)]\to  A(G)$ be the $^*$-isomorphism  such that 
\begin{equation}\label{Eq.Xi} 
p_G \Xi_{G}(a) p_G = a  
\end{equation}
  for all $a\in \pi_{G}[A(G)]$. This $^*$-isomorphism will be needed after we verify the properties of the maps $E'_{HG}$ defined in the following paragraph. 
	
	Fix $G\leq H\leq \Gamma$ and define  $E'_{HG}\colon A(H)\to \cB(\ell_2(\Gamma))$ by 
	\[
	E'_{HG}(a)= p_G a p_G. 
	\]
	This map is clearly linear and  $E'_{HG}(a)=\pi_G(a)$ for $a\in H_G$ and the system of maps $E'_{HG}$ is clearly commuting.

	We will now verify that     $\|E'_{HG}\|=1$. By the density of $C_{\Theta\rs H}H$ in $A(H)$ it suffices to check  (writing $E'=E'_{HG}$) $\|E'(a)\|\leq \|a\|$ for $a\in C_{\Theta\rs H} H$. Fix such $a$ and $\e>0$. Let $\xi\in \ell_2(G)$ be a unit vector which satisfies  $\|E'(a)\xi\|^2_2>\|E'(a)\|^2-\e$. 
	Since $G$ is a subgroup and $a\in C_{\Theta\rs H}H$, the vectors $(a-E'(a))\xi$ and $\xi$ are orthogonal and 
	\[
	\| a\|^2\geq \|(a-E'(a))\xi\|^2_2+\|E'(a)\xi\|_2^2>\|E'(a)\|^2-\e.
	\]
	 Since $a\in A(H)$ and  $\e>0$ were arbitrary, we have $\|E'\|\leq 1$. The converse inequality follows from  $E'(1)=1$.

	We prove that $E'_{HG}[A(H)]\subseteq \pi_G[A(G)]$. 
For $g\in G$ we have (with $\pi_G$ as in \eqref{Eq.piG}) $p_G u_g p_G=\pi_G(u_G)$ and by linearity $E'_{HG}(a)= \pi_G(a)$ for all $a\in C_{\Theta\rs G}G$. For $h\in H\setminus G$ we have $u_h \delta_g \perp \ell_2(G)$ for all $g\in G$ and therefore $E'_{HG}(u_h)=0$. 	If $G$ is a finite group, then $p_G$ has  finite rank and $C_{\Theta\rs G}G=A(G)$, hence these computations show that  $E'_{HG}(a)\in \pi_G[A(G)]$ for all $a\in A(H)$ and that $E'_{HG}(a)=\pi_G(a)$ for $a\in A(G)$. 
	In this case, let   (see \eqref{Eq.Xi})
	\[
	E_{HG}=\Xi_{G}\circ E'_{EG}. 
	\] 
By the already verified $\|E_{HG}\|=1$ and Tomiyama's theorem, in the case when $G$ is finite,  $E_{HG}$ is a conditional expectation.    

Now consider the case when  $G$ is not necessarily finite. Since $\Gamma$ is locally finite, so is $H$ and we can write it as an inductive limit of finite groups $H(\lambda)$, for $\lambda\in \Lambda$. With $G(\lambda)=G\cap H(\lambda)$, we have that $E_{H G(\lambda)}[A(H)]\subseteq A(G(\lambda))$ for every $\lambda$.  For $a\in A(H)$, the net $(E_{HH(\lambda)}(a))_\lambda$ is  Cauchy since it  converges to~$a$. By the contractivity and commutation,  $(E_{H G(\lambda)}(a))_\lambda$ is a Cauchy net. Define  
\[
E_{HG}(a)=\lim_\lambda E_{H G(\lambda)}(a)
\]
 Therefore $E_{HG}(a)\in \lim_\lambda A(G(\lambda))=A(G)$. Since $a\in A(H)$ was arbitrary, this  proves  $E_{HG}[A(H)]\subseteq A(G)$. It is clear that $E_{HG}(a)=a$ for $a\in A(G)$. Again, Tomiyama's theorem implies that $E_{HG}$ is a conditional expectation as required. 

	Clearly, the system of the maps $E_{HG}$ is commuting, and this completes the proof. 
\end{proof}

\begin{corollary}\label{C.Conditional} Suppose that $\Gamma$ is a locally finite group,   $(X,\Gamma,\Theta)$ is a CCR triple, and $G\leq \Gamma$. If $a\in \CCR_{(X,\Gamma,\Theta)}$ is such that the $\ell_2$-expansion $a=\sum_g \lambda_g  u_g$ (as in Lemma~\ref{L.ell2}) satisfies $\lambda_h=0$ for all  $h\in \Gamma\setminus G$,  then $a\in \cst(u_g\vert g\in G)$.  
\end{corollary}

\begin{proof} For $a$ as in the assumptions we clearly have $E_{\Gamma G}(a)=a$. 
\end{proof}

\section{Properties of CCR algebras $\CCR_{(X,\Gamma,\Theta)}$}
\label{S.Properties}
In this section we analyze the structure of \cstar-algebras of the form $\CCR_{(X,\Gamma,\Theta)}$ introduced in \S\ref{S.CCR}.  In the following lemma and elsewhere in this paper, $\otimes$ stands for the minimal tensor product. 

 \begin{lemma}\label{L.tensor}
 Suppose that $(Z,\Gamma,\Theta)$ is a CCR triple, that  $\Gamma$ is locally finite, that  $Z=X\sqcup Y$ for some nonempty sets $X$ and $Y$, and that $\Theta(\xi,\eta)=1$ if  $\xi\in X$ and $\eta\in Y$. Then (using the notation of \eqref{Eq.GammaF})
 \[
 \CCR_{(Z,\Gamma,\Theta)}\cong \CCR_{(X,\Gamma_X, \Theta\rs\Gamma_X)}\otimes 
\CCR_{(Y,\Gamma_Y,\Theta\rs\Gamma_Y)}.
\]  
 \end{lemma}
\begin{proof} The universal property and functoriality together imply that  $\CCR_{(Z,\Gamma,\Theta)}\cong
\CCR_{(X,\Gamma_X, \Theta\rs\Gamma_X)}\otimes_\alpha 
\CCR_{(Y,\Gamma_Y,\Theta\rs\Gamma_Y)}$ for some tensor product $\otimes_\alpha$. Lemma~\ref{L.AF} implies that these \cstar-algebras are AF, and therefore nuclear, hence $\otimes_\alpha$ is the minimal tensor product (see e.g., \cite{Black:Operator}). 
\end{proof}

In the following we follow von Neumann's convention and identify $2$ with $\{0,1\}$. 

\begin{lemma} \label{L.UHF} Suppose that $(\kappa\times 2,\Gamma,\Theta)$ is a CCR triple such that $\Gamma=\bigoplus_{\kappa\times 2} \bbZ/n\bbZ$ and  $\Theta$ satisfies (with $\lambda=\exp(2\pi i/n)$)
\[
\Theta(g(\alpha,i), g(\beta, j)) =\begin{cases}
\lambda, \text{ if $\alpha=\beta$, $i=0$, and $j=1$}, \\
1, \text{ otherwise.}
\end{cases}
\]
	Then $\CCR_{(\kappa,\Gamma,\Theta)}\cong \bigotimes_{\kappa} M_n(\bbC)$.  
	
	Moreover, if $\kappa$ is finite then every \cstar-algebra generated by unitaries $v_g$, $g\in \cG(\Gamma)$ that satisfy  relations $\cR(\Gamma,\Theta)$ is isomorphic to $\bigotimes_\kappa M_p(\bbC)$.\footnote{The conclusion holds without the finiteness assumption  on $\kappa$, but we will not need this fact.}  
\end{lemma}

\begin{proof}  If $F\Subset \kappa$, then $\Gamma_F=\bigoplus_{F\times 2} \bbZ/n\bbZ$ is a subgroup of $\Gamma$, and the restriction of $\Theta$ to $\Gamma_F$ entails that $\CCR_{(F,\Gamma_F,\Theta\rs \Gamma_F)}$ is generated by the  unitaries $v_\alpha=u_{\alpha,0}$ and $w_\alpha=u_{\alpha,1}$ such that $\cst(v_\alpha,w_\alpha)$ generate a copy of $M_n(\bbC)$, and that these copies commute. Since $\Gamma$ is locally finite, by Lemma~\ref{L.tensor} $\CCR(\Gamma_F,\bb\rs \Gamma_F)\cong \bigotimes_F M_n(\bbC)$.  Since $\Gamma$ is the direct limit of $\Gamma_F$, for $F\Subset \kappa$, by Lemma~\ref{L.AF} $\CCR_{(\kappa,\Gamma,\Theta)}$ is isomorphic to  $\bigotimes_\kappa M_n(\bbC)$ as required. 

We now prove that for a finite $\kappa$, every \cstar-algebra $A$ generated by unitaries $v_g$, $g\in \cG(\Gamma)$ that satisfy the relations $\cR(\Gamma,\Theta)$ is isomorphic to $\bigotimes_\kappa M_p(\bbC)$. The relations imply that $A$ is the closed linear span of $1$ and finite products of generators taken in the lexicographic order such that each $g(\alpha,j)$ occurs fewer than $p$ times. Thus, if $\kappa$ is finite, the dimension of $A$ is $p^{2\kappa}$---i.e., equal to the dimension of $\bigotimes_\kappa M_p(\bbC)$. By the universality of the latter, $A$ is isomorphic to it. 
\end{proof} 

\begin{lemma} 
Suppose  $(\kappa,\Gamma,\Theta)$ is a CCR triple. For every  $X\subseteq \Gamma$,\footnote{In this definition we are using the extension of $\Theta$ to $\Gamma$ defined in Lemma~\ref{L.CCR.1}.}   
\[
Z_{\Gamma,\Theta}(X)=\{g\in \Gamma\vert \Theta(x,g)=1\text{ for all }x\in X\}
\]
is a subgroup of $\Gamma$. 
\end{lemma}

\begin{proof} Suppose that $g$ and $h$ belong to $Z_{\Gamma,\Theta}(X)$. For $x\in X$ we have $\Theta(x,gh)=\Theta(x,g)\Theta(x,h)=1$. 

Similarly,  $\Theta(x,g^{-1})=\Theta(x,1)\overline{\Theta(x,g)}=1$. Since $x$ was arbitrary, this implies that $gh$ and $g^{-1}$ belong to $Z_{\Gamma,\Theta}(X)$. Since $g$ and $h$ were arbitrary, this proves that $Z_{\Gamma,\Theta}(X)$ is a subgroup of $\Gamma$. 
\end{proof}

If $A$ is a \cstar-subalgebra of $B$, then the relative commutant of $A$ in $B$ is 
\[
B\cap A'=\{b\in B\vert [a,b]=0\text{ for all }a\in A\}. 
\]

The following definition is taken from  \cite{FaKa:Nonseparable} (see also \cite[\S 7.4]{Fa:STCstar}). 
\begin{definition}
	A \cstar-subalgebra $B$ of a \cstar-algebra $A$ is \emph{complemented} in $A$ if $A=\cst(B,A\cap B')$.  
\end{definition}

The following is an analog of \cite[Lemma~10.2.10]{Fa:STCstar}. 

\begin{lemma} \label{L.commutant}
	Suppose that $(\kappa,\Gamma,\Theta)$ is a CCR triple,  $\Gamma$ is locally finite,  and   $X\subseteq \kappa$ is nonempty. Then the following are true. 
\begin{enumerate}
\item\label{1.commutant} 	$
	\CCR_{(\kappa,\Gamma,\Theta)}\cap \cst(\{u_g\vert g\in \Gamma_X\})'=\cst(\{u_g\vert g\in Z_{\Gamma,\Theta}(X)\})$. 
\item \label{2.commutant} $\cst(\{u_g\vert g\in \Gamma_X\})$ is complemented in $\CCR_{(\kappa,\Gamma,\Theta)}$ if and only if $\Gamma$ is generated by $\Gamma_X$ and $Z_{\Gamma,\Theta}(X)$. 
\end{enumerate}
\end{lemma}

\begin{proof} \eqref{1.commutant} Only the direct inclusion requires a proof. 

Let $B=\cst(\{u_g\vert g\in \Gamma_X\})$. 
Fix $a\in \CCR_{(\kappa,\Gamma,\Theta)}\cap B'$; we will prove that $a\in \cst(\{u_g\vert g\in Z_{\Gamma,\Theta}(X)\})$. By Lemma~\ref{L.ell2}, we can represent $a$ as a possibly infinite $\|\cdot\|_{2,\tau}$-convergent sum, $a=\sum_j \lambda_j u_{g(j)}$.

We will prove that  $u_{g(j)}\in Z_{\Gamma,\Theta}(\Lambda)$ for all $j$ such that $\lambda_j\neq 0$. Towards contradiction, assume that there is $j$ such that $\lambda_j\neq 0$ but $g(j)\notin Z_{\Gamma,\Theta}(\Lambda)$. Fix $h\in \Gamma_X$ such that $\eta=\Theta(h,g(j))$ is not equal to 1. Then $u_h u_{g(j)} u_h^*=\eta u_{g(j)}$, and therefore in the $\|\cdot\|_{2,\tau}$-expansion of  $u_h a u_h^*$, the $g(j)$-coefficient is equal to $\eta$. Therefore the $g(j)$-coefficient in the expansion of $a-u_h a u_h^*$ is $1-\eta\neq 0$, contradiction. 
	
	By Corollary~\ref{C.Conditional}, this implies that  $a\in \cst(u_g\vert g\in  Z_{\Gamma,\Theta}(\Lambda))$, as required. Since~$a$ was arbitrary, this proves the direct inclusion and completes the proof. 

\eqref{2.commutant} follows immediately from \eqref{1.commutant}. 
\end{proof}

The assumption that $\Gamma$ be locally finite used throughout \S\ref{S.Funct} and \S\ref{S.Properties} is probably unnecessary. 

\section{Non-uniqueness}

Our main objective in this section  is to prove Theorem~\ref{T.non-uniqueness} (a stronger result will be proven in \S\ref{S.non-classification}). 
The following is  \cite[Proposition~3.2]{FaKa:NonseparableII}. 

\begin{lemma}\label{L.K-theory} 
If   $A$ is a nonseparable AM algebra, then the following conditions hold. 
\begin{enumerate}
\item $K_0(A)=K_0(B)$ for every separable elementary submodel $B$ of $A$ (in symbols,  $B\prec A$). 
\item $A$ has a unique tracial state. 
\item $K_1(A)$ is trivial. \qed 
\end{enumerate}
\end{lemma}

The proofs of Theorem~\ref{T.non-uniqueness} and Theorem~\ref{T.many} use the set-theoretic notion of a closed unbounded (club) set. We recall the definitions, and direct the reader to  \cite[\S 6.2--6.4]{Fa:STCstar} for additional information. 

\begin{definition}
If $X$ is an uncountable set, then a family $\sfC$ of countable subsets of $X$ is club if the following two requirements are met.  
\begin{enumerate}
\item For every increasing sequence $Z_n$, for $n\in \bbN$, in $\sfC$ we have $\bigcup_n Z_n\in \sfC$. 
\item For every countable $Y\subseteq X$ there exists $Z\in \sfC$ such that $Y\subseteq Z$. 	
\pushcounter
\end{enumerate}
\end{definition}
\begin{definition}
  If $A$ is a nonseparable complete metric space, then a family $\sfC$ of separable closed subsets of $A$ is a club if the following two requirements are met. 
\begin{enumerate}
\popcounter 
\item For every increasing sequence $C_n$, for $n\in \bbN$, in $\sfC$ we have $\overline{\bigcup_n C_n}\in \sfC$. 
\item For every countable $B\subseteq A$ there exists $C\in \sfC$ such that $B\subseteq C$. 	
\pushcounter
\end{enumerate}
\end{definition}

We use the common jargon and say that `there are club many separable subspaces of $A$ with property $P$' if the set of separable subspaces of $A$ with property $P$ includes a club. By the Downwards L\"owenheim--Skolem Theorem, club many separable subspaces of a \cstar-algebras are \cstar-subalgebras. By \cite[Lemma~7.4.4]{Fa:STCstar}, if two nonseparable  \cstar-algebras $A$ and $B$ are isomorphic then club many separable \cstar-algebras of $A$ are complemented if and only if club many separable \cstar-subalgebras of $B$ are complemented.

By using functoriality (Proposition~\ref{P.funct}) we can express the following lemma whose proof is, being straightforward, omitted. 

\begin{lemma} \label{L.club} Suppose that  $(\kappa,\Gamma,\Theta)$ is a CCR triple such that $\kappa$ is uncountable and that $\sfC$ is a club of countable subsets of $\kappa$.  Then   the \cstar-subalgebras of the form $\CCR_{(X,\Gamma_X,\Theta\rs\Gamma_X)}$ for  $X\in \sfC$  form a club of separable substructures of $\CCR_{(\kappa,\Gamma,\Theta)}$. \qed  
\end{lemma}

\begin{theorem}\label{T.non-uniqueness}
Suppose that $A$ is a UHF algebra such that for some prime $p$ every element of $K_0(A)$ is divisible by $p$. Then for every uncountable cardinal $\kappa$ there exists an AM algebra of density character $\kappa$  with the same Elliott invariant as $A$ that is not UHF.\footnote{A much easier proof of this theorem (which would not lead to a proof of Theorem~\ref{T.many}) is described in the paragraphs following Question~\ref{Q.main}.} \end{theorem}

\begin{proof} Let $p(i)$, for $i\in \bbN$, be an enumeration of all primes $q$ such that $M_q(\bbC)$ unitally embeds into $A$  and let $\prod_i p(i)^{k(i)}$ be the generalized integer of $A$. Thus $k(i)\geq 1$ for all $i$. We may assume that $p=p(0)$ (the enumeration is not assumed to be in the increasing order). 

Let $J=\kappa\times 2\sqcup \{*\}$ and 
\begin{enumerate}
\item [(a)]	$\textstyle\Gamma=\bigoplus_{i>0} \bigoplus_{j<k(i)} (\bbZ/p(i)\bbZ)^2\times \bigoplus_J \bbZ/p\bbZ$. 
\item [(b)] Let $f_0(i,j)$, $f_1(i,j)$ for $i>0$ and  j$<k(i)$, $g(\alpha,i)$, for $\alpha<\kappa$, and $i<2$, and $g(*)$ be the generators of the corresponding direct summands of $\Gamma$. 
\item [(c)] Let $\cG_-=\{f_0(i,j),f_1(i,j)\vert i>0, j<k(i)\}$ and $\cG_+=\{g(\alpha,i)\vert \alpha<\kappa,i<2\}$. Then   $\cG=\cG_-\cup \cG_+\cup \{g(*)\}$ is a  generating set for   $\Gamma$.
\end{enumerate}
Let $\lambda=\exp(2\pi i/n)$. Let $\Theta\colon\cG^2\to \bbT$   be such that the following conditions hold. 
\begin{enumerate}
\item \label{1.bb} $\Theta(g(\alpha,0), g(\alpha,1))=\lambda$ for $\alpha<\kappa$. 
\item $\Theta(g(\alpha,1), g(\alpha,0))=\overline{\lambda}$, for $\alpha<\kappa$. 
\item $\Theta(g(\alpha,0), g(*))=\lambda$, for all $\alpha<\kappa$. 
\item $\Theta(g(*),g(\alpha,0))=\overline{\lambda}$, for all $\alpha<\kappa$. 
\item $\Theta(f_0(l,j),f_1(l,j))=\exp(2\pi/p(j))$, for $l>0$ and $j<k(l)$.
\item \label{5.bb} $\Theta(f_1(l,j),f_0(l,j))=\exp(-2\pi/p(j))$, for $l>0$ and $j<k(l)$.
\item $\Theta(g,h)=1$ for $g$ and $h$ in $\cG$ for which $\Theta$ hasn't been defined by \eqref{1.bb}--\eqref{5.bb}. 
\end{enumerate}
We first prove that $B=\CCR_{(\kappa,\Gamma,\Theta)}$ is AM. 
For $F\Subset\cG$ let $\Gamma_F$ denote the subgroup of $\Gamma$ generated by $F$.  By Proposition~\ref{P.funct}, $B_F=\cst(\Gamma_F)$  is  naturally identified with $\CCR_{(F,\Gamma_F, \Theta\rs\Gamma_F)}$. 

Since $B$ is the inductive limit of \cstar-subalgebras of the form $B_F$ for $F\subset \cG$, it will suffice to find a cofinal set $\cF$ of $F\Subset\cG$ such that $B_F$ is a full matrix algebra for every $F\in \cF$. 

Let $\cF$ be the set of all $F\Subset \cG$ such that   the following conditions hold
\begin{enumerate}
\item\label{1.F} For all $i>0$ and $j<k(i)$, $f_0(i,j)\in F$ if and only if $f_1(i,j)\in F$. 
\item \label{2.F} For all $\alpha<\kappa$, $g(\alpha,0)\in F$ implies $g(\alpha,1)\in F$ and there exists a unique $\beta=\beta(F)$ such that $g(\beta,1)\in F$ but $g(\beta,0)\notin F$. 
\item \label{3.F} $g(*)\in F$. 	
\end{enumerate}
Clearly for every $F_0\Subset \cG$ there is $F\in \cF$ such that $F_0\subseteq F$. 
We claim that $F\in \cF$ implies $B_F$ is a full matrix algebra. The group $\Gamma$ is clearly locally finite, and this fact will be used in the remaining part of this proof. With $F(-)=F\cap \cG_-$ and $F(+)={F\cap \cG_+}$,  Lemma~\ref{L.tensor} implies that $B_F\cong B_{F(-)}\otimes B_{F(+)}$. The `moreover' part of  Lemma~\ref{L.UHF} and    \eqref{1.F} together imply that $B_{F(-)}$ is isomorphic to a full matrix algebra.  It remains to prove that $B_{F(+)}$ is isomorphic to a full matrix algebra. We will prove this by replacing the generator $g(*)$ with $h\in \Gamma$ such that $B_{F(+)}=\cst((F(+)\setminus \{g(*)\})\cup \{h\})$ and the \cstar-algebra on the right-hand side is a full matrix algebra.  Conditions \eqref{2.F} and \eqref{3.F} imply that there is a finite subset $\alpha(j)$, for $j\leq m$ of $\kappa$ such that 
\[
F\cap\cG_+=\{g(\alpha(j),0), g(\alpha(j),1)\vert j<m\}\cup  \{g(\alpha(m),0)\}.  
\]
Denoting the first set on the right-hand side by $G$, the `moreover' part of Lemma~\ref{L.UHF} implies that $B_G$ is a full matrix algebra. Let 
\[
\textstyle h=g(*)\prod_{j<m} g(\alpha(j),0) . 
\]
This is a product of generators, each of order $p$, and therefore it has order $p$. Moreover, for every $g\in \cG$ we have 
\[
\textstyle \Theta(g,h)=\prod_{j<m} \Theta(g,g(\alpha(j),0))\cdot \Theta(g,g(*)). 
\]
This implies that $\Theta(g(\alpha(m),0),h)=\lambda$ and that $\Theta(g(\alpha(j),i),h)=1$ for all $j<m$ and $i<2$. Also, $g(*)=h\prod_{j<m} g(\alpha(j),0)^{p-1}$ hence $B_F=\cst((F\setminus\{g(*)\})\cup \{h\})$.  Therefore Lemma~\ref{L.tensor} implies $B_{F(+)}$ is isomorphic to a full matrix algebra and concludes the proof that $B$ is AM. 

Towards proving  that $B$ is not a tensor product of full matrix algebras we first  prove that $B$ does not have club many complemented separable \cstar-subalgebras. Since $\cG_-$ is countable and $\kappa$ is uncountable, by using  Lemma~\ref{L.club} it will suffice to prove that 
 
 \begin{claim} If $\cG_-\cup \{g(*)\}\subseteq F\subseteq \cG$ and $F$ is countable, then $B_F$ is not complemented in $B$. 	
 \end{claim}

\begin{proof} 
We claim that (writing an element of $\Gamma$ as a product of generators  $\prod_{i<m} h(i)$ in the reduced form) 
\begin{multline*}
\textstyle Z_{\Gamma,\Theta}(F)=\langle \prod_{i<m} h(i)\vert m\in \bbN,\\
\textstyle h(i)\in \cG\setminus F, \text{ and } \prod_{i<m} \Theta(h(i), g(*))=1\}. 
\end{multline*}
Only the direct inclusion requires a proof. 
Suppose that $\prod_{i<m} h(i)$ does not belong to the right-hand side. We consider two possible cases. 

Suppose first that   $h(r)\in F$ for some $r<m$. The word $\prod_{i<m} h(i)$ is in the reduced form and  the generators in $F\setminus \{g(*)\}$ come in non-commuting pairs ($f_0(l,j)$ and $f_1(l,j)$, $g(\alpha,0)$ and $g(\alpha,1)$). Therefore 
\[
\textstyle \prod_{i<m} h(i)=h(r)^d h,
\]
 where $h$ is a product of generators distinct both from $h(r)$ and its pair, and $h(r)^d\neq 1$.  Suppose in addition that $h(r)\neq g(*)$. This implies that $h(r)^d h$ does not commute with the generator paired with $h(r)$. 
Now assume that $h(r)=g(*)$. Since $F$ is infinite, there exists $g(\alpha,1)\in F$ such that $g(\alpha,j)$ for $j<2$   is not a factor of $h$. Then  $\Theta(g(\alpha,1),g(r)^d h)=\lambda^d\neq 1$. 

It remains to analyze the remaining case, when $h(i)\notin F$ for all $i<m$ but $\prod_{i<m} \Theta(h(i), g(*))\neq 1$. This clearly implies that $g(*)$ does not commute with $\prod_{i<m} h(i)$.

This proves that $Z_{\Gamma,\Theta}(F)$ is as claimed. Since $\Gamma$ is locally finite, by Lemma~\ref{L.commutant}, $B_F$ is complemented in $\CCR_{(\kappa,\Gamma,\Theta)}$ if and only if $\Gamma$ is generated by $F\cup Z_{\Gamma,\Theta}(F)$. Since $\kappa$ is uncountable, we can find $\alpha< \kappa$ such that $g(\alpha,0)\notin F$. Then $\Theta(g(\alpha,0),g(*))=\lambda\neq 1$, and $g(\alpha,0)\notin \Lambda$. 
\end{proof}

Since $\kappa$ is uncountable, the countable subsets of $\cG$ that include $\cG_-\cup \{g(*)\}$ form a club, and Claim implies that club many separable \cstar-subalgebras of $\CCR_{(\kappa,\Gamma,\Theta)}$ are not complemented. In a UHF algebra, club many separable \cstar-subalgebras are complemented (\cite[Example~7.4.2]{Fa:STCstar}). By \cite[Lemma~7.4.4]{Fa:STCstar}, if two \cstar-algebras $A$ and $B$ are isomorphic then club many separable \cstar-algebras of $A$ are complemented if and only if club many separable \cstar-subalgebras of $B$ are complemented,  and therefore  $B$  is not UHF.  
\end{proof}

\section{Non-classification}
\label{S.non-classification}

In this section we prove our main non-classification result. 

\begin{theorem}\label{T.many} Suppose that $A$ is a UHF algebra such that for some prime $p$ every element of $K_0(A)$ is divisible by $p$. Then for every uncountable cardinal $\kappa$ there exists $2^\kappa$ nonisomorphic AM algebras of density character $\kappa$ with the same Elliott invariant as $A$. 
\end{theorem}

The proof of this result uses basic continuous model theory and Shelah's non-structure theory (\cite{shelah2019general}, \cite{shelah2000non}), and we assume that the reader is familiar with the former (see e.g., \cite[Appendix C]{Fa:STCstar} or \cite{Muenster}); the latter (incomparably more complex) ingredient will be treated as a blackbox.   The following is an analog of Shelah's  order property (OP), first used in the context of \cstar-algebras in \cite{FaHaSh:Model1}.
\begin{definition}
Suppose that  $\varphi(\bar x, \bar y)$ is a formula of the language of \cstar-algebras in $2n$ variables, in which $\bar x$ and $\bar y$ are $n$-tuples of the same sort.  Define a binary relation $\prec_\varphi$ on $A_1^n$  by letting $\bar a\prec_\varphi \bar b$ if and only if $\bar a$ and $\bar b$ are of the appropriate sort and the following holds 
\[
\varphi^A(\bar a,\bar b)=1\text{ and } \varphi^A(\bar b, \bar a)=0.
\]
It is not required that $\prec_\varphi$ is transitive (but it is clearly antisymmetric). A \emph{$\varphi$-chain} is an indexed set $\bar a_x$, for $x\in I$, where $(I,<)$ is a linear ordering and $x<y$ if and only if $\bar a_x\prec_\varphi \bar a_y$. 
\end{definition}

Lemma~\ref{L.nonclassification} below is an immediate consequence of Shelah's non-structure results, and we sketch a proof for reader's convenience. It should be emphasized that, in spite of being stated as a result about \cstar-algebras, this result is about general metric structures. 

\begin{lemma} \label{L.nonclassification} Suppose that $\kappa$ is an uncountable cardinal, $\cK$ is a subcategory of \cstar-algebras,  and there exist a quantifier-free formula $\varphi$ and a functor $\bbF$  from the category of linear orderings into $\cK$ such that for every $J$,  $\bbF(J)$ is in $\cK$ and it is  generated by a $\varphi$-chain $\bar a_x$, for $x\in J$. Then for every uncountable cardinal $\kappa$ there are $2^\kappa$ nonisomorphic \cstar-algebras of density character $\kappa$ in $\cK$.  

In addition, for every separable \cstar-algebra $A$ there are $2^\kappa$ nonisomorphic \cstar-algebras of the form $A\otimes B$ for $B\in \cK$ of density character~$\kappa$. 
\end{lemma}

The assumption that $\varphi$ is quantifier-free can be replaced by the assumption that for $I\subseteq J$ the inclusion of $\bbF(I)$ into $\bbF(A)$ is sufficiently elementary (more precisely, that it preserves the values of $\varphi$). 

The proof of Lemma~\ref{L.nonclassification} uses the the following definition and  Lem\-ma~\ref{L.FaKa} stated below.

\begin{definition} Suppose that $n\geq 1$ and $\varphi(x,y)$ is a $2n$-ary formula in which $\bar x$ and $\bar y$ are of the same sort. A $\varphi$-chain $\cC=(\bar a_x\vert x\in I)$ in a \cstar-algebra $A$ is \emph{weakly $(\aleph_1,\varphi)$-skeleton like inside $A$} if for every $\bar a\in A^n$ there is a countable $I_{\bar a}\subseteq I$ with the following property.  If $x$ and $y$ are in $I$  and such that  $\bar a_x\prec_\varphi \bar a_y$ and no $z\in I_{\bar a}$ satisfies $\bar a_x\prec_\varphi\bar a_z \prec_\varphi \bar a_y$, then $\varphi^A(\bar a_x,\bar  a)=\varphi^A(\bar a_y,\bar a)$ and $\varphi^A(\bar a, \bar a_x)=\varphi^A(\bar a,\bar a_y)$. 
\end{definition}

The following  is \cite[Lemma~6.4]{FaKa:NonseparableII}, proved by a heavy use of the results of~  \cite{FaSh:Dichotomy}. 

\begin{lemma} \label{L.FaKa} Suppose that $\cK$ is a subcategory of \cstar-algebras, $\varphi(x,y)$ is a $2n$-ary formula in which $\bar x$ and $\bar y$ are of the same sort,  and $\kappa$ is an uncountable cardinal. If for every linear ordering $\Lambda$ of cardinality $\kappa$ there is $B_\lambda\in \cK$ of density character $\kappa$ such that $B^n_1$ includes a $\varphi$-chain isomorphic to $\Lambda$ which is weakly $(\aleph_1,\varphi)$-skeleton like, then $\cK$ contains $2^\kappa$ nonisomorphic \cstar-algebras of density character $\kappa$.  \qed 
\end{lemma}

\begin{proof}[Proof of Lemma~\ref{L.nonclassification}] We first   verify that the $\varphi$-chain $\bar a_x$, for $x\in J$, is weakly $(\aleph_1,\varphi)$-skeleton like in $A=\bbF(J)$. To prove this, with $n$ such that $\varphi$ is $2n$-ary fix $\bar b\in A^n$ (of the appropriate sort). By the Downwards L\"owenheim--Skolem Theorem, there is a countable  $J_{\bar b}\subseteq J$ such that $\bar b$ belongs to $\bbF(J_{\bar n})$ (using functoriality to give meaning to this formula). Fix $x$ and $y$ such that $\bar a_x\prec_\varphi \bar a_y$ and no $z\in J_{\bar b}$ satisfies $\bar a_x\prec_\varphi\bar a_z \prec_\varphi \bar a_y$. By the functoriality of $\bbF$, there is then an isomorphism 
\[
\Phi\colon \bbF(J_{\bar b}\cup \{x\})
\to \bbF(J_{\bar b}\cup \{y\})
\]
that extends the identity on $\bbF(J_{\bar b})$ and sends $\bar a_x$ to $\bar a_y$. 

Since $\varphi$ is quantifier-free, this implies  $\varphi^A(\bar a_x,\bar  b)=\varphi^A(\bar a_y,\bar b)$ and $\varphi^A(\bar b, \bar a_x)=\varphi^A(\bar b,\bar a_y)$. If $J$  is uncountable, then the density character of $\bbF(J)$ is equal to $|J|$, and Lemma~\ref{L.FaKa} implies that there are $2^\kappa$ nonisomorphic algebras of the density character $\kappa$ and the prescribed $K$-theory.

For the second part, we only need to prove  that if a chain $\bar a_x$, for $x\in J$, is weakly $(\aleph_1,\varphi)$-skeleton like in $\bbF(J)$ then it is weakly  $(\aleph_1,\varphi)$-skeleton like in $B\otimes \bbF(J)$. Since $\varphi$ is quantifier-free, it is still a $\varphi$-chain. Let  $\bar b$ and $J_{\bar a}$ be as in the first part of the proof. Mimicking this proof, while replacing separable algebras of the form  $\bbF(J')$ in this proof with $B\otimes \bbF(J')$, one concludes the proof. 
\end{proof}

\begin{proof}[Proof of Theorem~\ref{T.many}] In order to use Lemma~\ref{L.nonclassification}, we define a functor that sends a linear ordering $J$ to an AM \cstar-algebra $A=\bbF(J)$. Fix~$J$. Suppose that the generalized integer associated with $A$ is $\textstyle n_A=\prod_j p(j)^{k(j)}$, let $\Gamma_A=\bigoplus_j \bigoplus_{l<k(j)} \bbZ/p(j)\bbZ$ and $X_A=\bigcup_j k(j)$. With $p$ as in the assumption of Theorem~\ref{T.many}, let 
$\Gamma_J=\bigoplus_{x\in J} \bigoplus_{i<2} \bbZ/p\bbZ$ and 
\[
\textstyle \Gamma=\Gamma_A\oplus \Gamma_J. 
\]
Let $g(x,j)$, for $x\in J$ and $j<2$,  denote the generators of $\Gamma_J$. 
Define~$\Theta$ on the generators of $\Gamma$  as follows. 
	\begin{enumerate}
	\item\label{1.main} The restriction of $\Theta$ to $\Gamma_A$ is such that $\CCR_{(X_A,\Gamma_A,\Theta\rs \Gamma_A)}\cong A$ (as in the proof of Theorem~\ref{T.non-uniqueness}). 
	\item \label{2.main} $\Theta(g(x,0),g(y,1))=\lambda=\overline{\Theta(g(y,1), g(x,0))}$ if $x$ and $y$ are in $J$ and $x\leq y$.
	\item $\Theta(g,h)=1$ for all generators for which $\Theta(g,h)$ is not determined by \eqref{1.main} or \eqref{2.main}.    	
	\end{enumerate}
Let 
\[
\bbF(J)=\CCR_{(X_A\cup J, \Gamma, \Theta)}. 
\]
The functoriality of $\bbF$ follows by Proposition~\ref{P.funct}. 
Consider the following formula
\[
\varphi(a,b,c,d)=\frac 12 \|ad-da\|.
\]
Then  by \eqref{2.main} we have  
$\varphi^{\bbF(J)}(u_{g(x,0)}, u_{g(x,1)},u_{g(y,0)}, u_{g(y,1)})=1$ if $x\leq y$ and $0$ otherwise. 
Therefore  $(u_{g(x,0)}, u_{g(x,1)})$, for $x\in J$, form a $\varphi$-chain isomorphic to $J$. Since these unitaries generate $\bbF(J)$, by Proposition~\ref{P.funct} the assumptions of Lemma~\ref{L.nonclassification} are satisfied. Therefore for every uncountable $\kappa$ there are $2^\kappa$ nonisomorphic  algebras of the form $A\otimes \bbF(J)$ of density character $\kappa$. 

\begin{claim} For every linear order $J$, the algebra $\bbF(J)$ is AM.  
\end{claim}

\begin{proof} As before, for $F\Subset J$ let $B_F=\cst(\{g(x,i)\vert x\in F, i<2\})$.  Since $\bbF(J)$ is the inductive limit of such $B_F$, it will suffice to prove that  $B_F$ is isomorphic to $M_{p^{|F|}}(\bbC)$ for every $F\Subset J$.    

Fix $F\Subset J$ and let $x(j)$, for $j<k$, be its increasing enumeration. For $j<k$ and $i<2$ define $g'(j,i)$ as follows. 
\begin{enumerate}
\item $g'(x(j),0)=g(x(j),0)$ for all $j<k$.  
\item $g'(x(0), 1)=g(x(0),1)$. 
\item $g'(x(j),1)=g(x(j),1)\prod_{k<j} g(x(k), 1)$, for $1\leq j<k$. 
\end{enumerate}
Then $g'(x(j), i)$, for $j<k$ and $i<2$, generate $B_F$. Also, \eqref{2.main} implies that  $g'(x(j),i)$ and $g'(x(j'),i')$ commute when $j\neq j'$, while 
\[
g'(x(j),0) g'(x(j),1)=\lambda 
g'(x(j),1) g'(x(j),0)
\]
for all $j<k$. Since $\Gamma$ is locally finite, the `moreover' part of Lemma~\ref{L.UHF} now implies that $B_F$ is a full matrix algebra, as required. 
\end{proof} 

Every AM algebra is, being a unital inductive limit of monotracial \cstar-algebras,  monotracial.  Lemma~\ref{L.K-theory} implies that each  $\bbF(J)$ has the  same $K_0$ and $K_1$ groups as $A$, and therefore the same Elliott invariant as $A$. This completes the proof. 
 \end{proof}

\section{Concluding remarks}

In the present paper we explored constructions of nonseparable CCR algebras from bicharacters on uncountable direct sums of cyclic groups. Since the set-theoretic study of infinite abelian groups has a long and distinguished history (e.g., \cite{shelah1974infinite}, \cite{eklof1997set}, \cite{magidor1994does}), bicharacters on more interesting groups could lead to even more intriguing examples of nonseparable \cstar-algebras.  An alternative route towards our results worth exploring (suggested by Ilan Hirshberg) could proceed via uncountable spin systems (see \cite{arveson2003structure}).   

Another related source of nonseparable \cstar-algebras are `twisted' cocycle crossed products by an uncountable discrete group  (\cite{packer1989twisted}), already used  (implicitly) in \cite[\S 6]{FaKa:Nonseparable} (see \cite[Remark 6.9]{FaKa:Nonseparable}).

For the definitions of classifiable \cstar-algebras and Elliott invariant see e.g.,  \cite{Ror:Classification}.

\begin{question} \label{Q.main} Suppose that $A$ is a nuclear, simple, separable \cstar-algebra and $\kappa$ is an uncountable cardinal. 
\begin{enumerate}
\item \label{1.Q.Main} Is it true that there are $2^\kappa$ nonisomorphic simple nuclear  \cstar-algebras of density character $\kappa$ with the Elliott invariant equal to $\Ell(A)$? 
\item \label{2.Q.Main} Is it true that there are at least two nonisomorphic simple nuclear \cstar-algebras of density character $\kappa$ with the Elliott invariant equal to $\Ell(A)$? 
\item\label{3.Q.main} Is it true that there exists a  simple nuclear \cstar-algebras of density character $\kappa$ with the Elliott invariant equal to $\Ell(A)$? 
\end{enumerate}
\end{question}

If $A$ is as in Question~\ref{Q.main}, $D$ is  a strongly self-absorbing \cstar-algebra (see \cite{ToWi:Strongly}, but all we need is that $\bigotimes_{\aleph_0}D\cong D$) and $A$   tensorially absorbs $D$,  then Question~\ref{Q.main} \eqref{3.Q.main} has a positive answer. The algebra $A_1=A\otimes \bigotimes_\kappa D$ has the same Elliott invariant as $A$ and density character $\kappa$. This is because by the assumption on $D$ club many separable subalgebras of $A_1$ are isomorphic to $A\otimes D\cong A$, and by the reflection argument as in Lemma~\ref{L.K-theory}, the Elliott invariant of $A_1$ is equal to that of $A$. In particular, Question~\ref{Q.main} \eqref{3.Q.main} has a positive answer for  every classifiable \cstar-algebra which absorbs the Jiang--Su algebra $\cZ$. We don't know the answer to Question~\ref{Q.main} \eqref{3.Q.main} for  R\o rdam's simple \cstar-algebra with both finite and infinite projections (\cite{Ror:Simple}) or for Villadsen's algebras with perforated K-theory groups (see \cite{Ror:Classification}). 

If $A$ tensorially absorbs both $\cZ$ and some full matrix algebra $M_n(\bbC)$ (e.g., if $A$ is a  UHF algebra such that for some prime $p$ every element of $K_0(A)$ is divisible by $p$, as in the assumptions of Theorem~\ref{T.non-uniqueness}),  then Question~\ref{Q.main} \eqref{2.Q.Main} has a positive answer, as witnessed by  $A\otimes \bigotimes_\kappa \cZ$ and $A\otimes\bigotimes_\kappa M_n(\bbC)$  (a proof of this is analogous to the proofs in \cite[\S 3]{FaKa:Nonseparable}). 

	An alternative approach to Question~\ref{Q.main} \eqref{2.Q.Main} proceeds by finding a \cstar-algebra with the same invariant as a given $\cZ$-absorbing $A$ that is tensorially prime (i.e., not isomorphic to the tensor product of two infinite-dimensional \cstar-algebras). This was done for some AF and some  purely infinite, simple \cstar-algebras in \cite[Theorem~B (1) and (4)]{suzuki2021rigid}.   In Theorem C of the same paper, a monotracial, tensorially prime, nuclear, simple \cstar-algebra with uncountable $K_1$-group was constructed. \cstar-algebras constructed in  \cite{suzuki2021rigid} are  associated with iterated wreath products of groups and they even have no nontrivial central sequences. Every infinite-dimensional separable AM algebra has nontrivial central sequences (for much stronger results see \cite{ando2016non} and \cite{enders2021commutativity}).   The first examples of (necessarily nonseparable) AM algebras with this property  were constructed in \cite{FaHaKaTi} using weakenings of the  Continuum Hypothesis. These algebras have the same Elliott invariant as $M_{2^\infty}$. 

Apparently the simplest instance of Question~\ref{Q.main} \eqref{1.Q.Main} not resolved by our Theorem~\ref{T.1}  is the case of $\bigotimes_{p\text{ prime}} M_p(\bbC)$. 
Probably the most interesting case is when $A$ is the Jiang--Su algebra~$\cZ$, whose Elliott invariant is equal to that of the complex numbers (see \cite{JiangSu}, also \cite{ghasemi2019strongly} and \cite{schemaitat2019jiang} for most recent treatments).

A programme analogous to one given by  Question~\ref{Q.main} was pursued in \cite{vaccaro2017trace} and \cite{suri2017naimark}. In these papers the authors studied the variety of the invariants of counterexamples to Naimark's Problem (using Jensen's diamond).

Graph CCR algebras provided the first example of a nuclear, simple \cstar-algebra with the property that its automorphism group did not act transitively on its pure state space (\cite{Fa:Graphs}, see \cite[\S 9.4]{Fa:STCstar}). It is likely that the CCR algebras associated with bicharacters on uncountable abelian groups will find other uses. For example, every graph CCR algebra is a complexification of an operator algebra on a real Hilbert space, and therefore isomorphic to its opposite algebra. The only known examples of simple, nuclear, \cstar-algebras not isomorphic to their opposites were constructed using Jensen's diamond in \cite{farah2016simple}. It is possible that, with an appropriately chosen group and a bicharacter,  the CCR algebras considered in this paper may provide a ZFC example. 

Our examples can also be used to construct $2^\kappa$ nonisomorphic hyperfinite II$_1$ factors with predual of density character $\kappa$, but this has already been proved in \cite{FaKa:NonseparableII} (see also  \cite{suzuki2021rigid}). However, the  II$_1$ factors obtained from our CCR algebras may have other interesting properties. For example, it is not known whether a hyperfinite II$_1$ factor not isomorphic to its opposite can be constructed in ZFC (see \cite{farah2020rigid} for a construction using Jensen's diamond). 

\section*{Acknowledgments}
Partially supported by NSERC. This paper is a part of the second author's PhD thesis. We would like to thank the anonymous referee for  useful comments and to Pavlos Motakis for very helpful advice on Banach spaces.

\end{document}